\documentclass{amsart}
\usepackage{hyperref}
\usepackage{amsmath,amsfonts,amsthm, amssymb, bbm, color}
\usepackage[T1]{fontenc}
\usepackage{graphicx}

\theoremstyle{plain}
\newtheorem{theorem}{Theorem}
\newtheorem{lemma}[theorem]{Lemma}
\newtheorem{corollary}[theorem]{Corollary}
\newtheorem{proposition}[theorem]{Proposition}

\newtheorem{definition}[theorem]{Definition}

\theoremstyle{definition}
\newtheorem{example}[theorem]{Example}
\newtheorem{remark}[theorem]{Remark}

\theoremstyle{remark}

\newcommand{\cB}{{\mathcal B}}

\newcommand{\cF}{{\mathcal F}}

\newcommand{\cI}{{\mathcal I}}

\newcommand{\cW}{{\mathcal W}}
\newcommand{\cX}{{\mathcal X}}

\newcommand{\bbE}{\mathbb{E}}

\newcommand{\bbN}{\mathbb{N}}

\newcommand{\bbP}{\mathbb{P}}

\newcommand{\bbR}{\mathbb{R}}

\newcommand{\bbZ}{\mathbb{Z}}

\newcommand{\eps}{\varepsilon}
\newcommand{\ind}{\mathbbm{1}}
\newcommand{\bsl}{\backslash}
\newcommand{\supp}{\mathop{\mathrm{supp}}\nolimits}

\newcommand{\Var}{\mathop{\mathrm{Var}}\nolimits}

\makeatletter
\let \@fnsymbol\@arabic
\makeatother

\begin{document}

\title
[Ergodic decompositions of stationary max-stable processes]
{Ergodic decompositions of stationary max-stable processes in terms of their spectral functions}

\author{Cl\'ement Dombry}

\address{Univ.\ Bourgogne Franche-Comt\'e, Laboratoire de Math\'ematiques de Besan\c con, UMR CNRS 6623, 16 route de Gray,
25030 Besan\c con cedex, France}
\email{Email: clement.dombry@univ-fcomte.fr}

\author{Zakhar Kabluchko}

\address{Universit\"at M\"unster, Institut f\"ur Mathematische Statistik, Orl\'{e}ans-Ring 10, 48149 M\"unster, Germany} \email{zakhar.kabluchko@uni-muenster.de}

\begin{abstract}
We revisit conservative/dissipative and positive/null decompositions of stationary max-stable processes. Originally, both decompositions were defined in an abstract way based on the underlying non-singular flow representation. We provide simple criteria which allow to tell whether a given spectral function belongs to the conservative/dissipative or positive/null part of the de Haan spectral representation. Specifically, we prove that a  spectral function is null-recurrent iff it converges to $0$ in the Ces\`{a}ro sense. For processes with locally bounded sample paths we show that a spectral function is dissipative iff it converges to $0$. Surprisingly, for such processes a spectral function is integrable a.s.\ iff it converges to $0$ a.s. Based on these results, we provide new criteria for ergodicity, mixing, and existence of a mixed moving maximum representation of a stationary max-stable process in terms of its spectral functions. In particular, we study a decomposition of max-stable processes which characterizes the mixing property.
\end{abstract}

\keywords{max-stable random process, de Haan representation,  non-singular flow, conservative/dissipative decomposition, positive/null decomposition, ergodic process, mixing process, mixed moving maximum process}
\subjclass[2010]{Primary: 60G70;  Secondary: 60G52, 60G60, 60G55, 60G10, 37A10, 37A25}
\maketitle

\section{Statement of main results}
\subsection{Introduction}
A stochastic process $(\eta(x))_{x\in \cX}$ on $\mathcal{X}=\mathbb{Z}^d$ or $\mathcal{X}=\mathbb{R}^d$ is called \textit{max-stable} if
\[
\frac 1 n \bigvee_{i=1}^n \eta_i \stackrel{f.d.d.}{=}\eta \quad \mbox{for all } n\geq 1,
\]
where $\eta_1, \ldots,\eta_n$ are i.i.d.\ copies of $\eta$, $\bigvee$ is the pointwise maximum, and $\stackrel{f.d.d.}{=}$ denotes the equality of finite-dimensional distributions. Max-stable processes arise naturally when considering limits for normalized pointwise maxima of independent and identically distributed (i.i.d.)\ stochastic processes and hence play a major role in spatial extreme value theory;
see, e.g.,\ de Haan and Ferreira \cite{dHF06}.
We restrict our attention to processes with non-degenerate (non-constant) margins.
The above definition implies that the marginal distributions of $\eta$ are $1$--Fr\'echet, that is
$$
\bbP[\eta(x)\leq z]={\rm{e}}^{-c(x)/z} \quad \mbox{for all }  z>0,
$$
where $c(x)>0$ is a scale parameter.   

A fundamental representation theorem by de Haan \cite{dH84} states that any stochastically continuous max-stable process $\eta$ can be represented (in distribution) as
\begin{equation}\label{eq:deHaan}
\eta(x)=\bigvee_{i\geq 1} U_iY_i(x),\quad x\in \mathcal{X},
\end{equation}
where
\begin{itemize}
\item[-] $(U_i)_{i\geq 1}$ is a decreasing enumeration of the points of a Poisson point process on $(0,+\infty)$ with intensity measure $u^{-2}\mathrm{d}u$,
\item[-] $(Y_i)_{i\geq 1}$, which are called the \textit{spectral functions},  are i.i.d.\ copies of a non-negative process $(Y(x))_{x\in \mathcal{X}}$ such that
$\bbE[Y(x)] < +\infty$ 
for all $x\in \mathcal{X}$,
\item[-] the sequences $(U_i)_{i\geq 1}$ and $(Y_i)_{i\geq 1}$ are independent.
\end{itemize}

In this paper, we focus  on \emph{stationary} max-stable processes that play an important role for modelling purposes; see, e.g.,\ Schlather~\cite{S02}.
The structure of stationary max-stable processes was first investigated by  de Haan and Pickands~\cite{dHP86} who related them to non-singular flows (which are referred to as ``pistons'' in~\cite{dHP86}). Using the analogy between max-stable and sum-stable processes and the works of Rosi\'nski \cite{R95,R00}, Rosi\'nski and Samorodnitsky \cite{RS96} and Samorodnitsky \cite{S04,S05} on sum-stable processes,
the representation theory of stationary max-stable processes via non-singular flows was developed by Kabluchko \cite{K09},
Wang and  Stoev \cite{WS10,WS10_association},  Wang {\it et al.}\ \cite{WRS13}.
In these papers, the conservative/dissipative (or Hopf) and positive/null (or Neveu) decompositions from  non-singular ergodic theory were used to introduce the corresponding decompositions $\eta = \eta_C\vee \eta_D$ and $\eta = \eta_P\vee \eta_N$
of the  stationary max-stable process. These definitions were rather abstract (see Sections~\ref{app:CD} and~\ref{app:PN} where we shall recall them) and did not allow to distinguish between conservative/dissipative or positive/null cases by looking just at the spectral functions $Y_i$ from the de Haan representation~\eqref{eq:deHaan}.  The purpose of this paper is to provide a \textit{constructive} definition of these decompositions. Our main results in this direction can be summarized as follows.
In Section~\ref{app:CD} we shall prove that in the case when the sample paths of $\eta$ are a.s.\ locally bounded, a spectral function $Y_i$ belongs to the dissipative (=mixed moving maximum) part of the process if and only if $\lim_{x\to\infty} Y_i(x) = 0$.  The class of  locally bounded processes is sufficiently general for applications.  On the other hand, the assumption of local boundedness cannot be removed; see Example~\ref{eq:loc_bound}. In Section~\ref{app:PN} we shall prove that a spectral function $Y_i$ belongs to the null (=ergodic) part if and only if it converges to $0$ in the Ces\`{a}ro sense. In Section~\ref{sec:mixing}, we shall introduce one more decomposition which characterizes mixing.

\subsection{Ergodic properties of max-stable processes}\label{subsec:ergodic_properties}
Our results can be used to give new criteria for ergodicity, mixing, and existence of mixed moving maximum representation of max-stable processes. These criteria extend and simplify the results of Stoev~\cite{S08}, Kabluchko and Schlather~\cite{KS10} and Wang {\it et al.}\ \cite{WRS13}.

In the following, $(\eta(x))_{x\in \cX}$ denotes a stationary, stochastically continuous max-stable process on  $\mathcal{X}=\mathbb{Z}^d$ or $\mathbb{R}^d$ with de Haan representation~\eqref{eq:deHaan}. In the case when $\mathcal{X}=\mathbb{R}^d$, the process $Y$ is continuous in $L^1$ by Lemma~2 in~\cite{dH84}.
Since continuity in $L^1$ implies stochastic continuity and since every stochastically continuous  process has a measurable and separable version, we shall tacitly assume throughout the paper that both $\eta$ and $Y$ are measurable and separable processes. These assumptions (as well as the assumption of stochastic continuity) are empty (and can be ignored) in the discrete case $\mathcal{X}=\mathbb{Z}^d$.

Our first result is a characterization of ergodicity. Let $\lambda(\mathrm{d}x)$ be the counting measure on $\mathbb{Z}^d$ (in the discrete-time case) or the Lebesgue measure on $\mathbb{R}^d$ (in the continuous-time case), respectively. For $r>0$, write $B_r=[-r,r]^d\cap \mathcal{X}$.
\begin{theorem}\label{theo:ergodicity}
For a stationary, stochastically continuous max-stable process $\eta$ the following conditions are equivalent:
\begin{itemize}
\item[(a)] $\eta$ is ergodic;
\item[(b)] $\eta$ is weakly mixing;
\item[(c)] $\eta$ has no positive recurrent component in its spectral representation, that is $\eta_P=0$;
\item[(d)] $\lim_{r\to\infty} \frac 1{\lambda(B_r)} \int_{B_r} \mathbb{E} [Y(x)\wedge Y(0)] \lambda(\mathrm{d} x) = 0$;
\item[(e)] $\lim_{r\to\infty} \frac 1{\lambda(B_r)} \int_{B_r} Y(x) \lambda(\mathrm{d} x) = 0$ in probability;
\item[(f)] $\liminf_{r\to\infty} \frac 1{\lambda(B_r)} \int_{B_r} Y(x) \lambda(\mathrm{d} x) =0$ almost surely.
\end{itemize}
\end{theorem}
The equivalence of (a), (b), (c), (d) in Theorem~\ref{theo:ergodicity} was known before (see Theorem~3.2 in~\cite{KS10} for the equivalence of (a), (b), (d) in the case $d=1$,  Theorem~8 in~\cite{K09} for the equivalence of (a) and (c) in the case $d=1$, and Theorem~5.3 in~\cite{WRS13} for an extension to the $d$-dimensional case). We shall prove in Section~\ref{app:CD} that (c), (e), (f) are equivalent by exploiting a new characterization of the positive/null decomposition.

The next theorem characterizes mixing (which is a stronger property than ergodicity).
\begin{theorem}\label{theo:mixing}
For a stationary, stochastically continuous max-stable process $\eta$ the following conditions are equivalent:
\begin{itemize}
\item[(a)] $\eta$ is mixing;
\item[(b)] $\eta$ is mixing of all orders;
\item[(c)] $\lim_{x\to\infty} \mathbb{E} [Y(x)\wedge Y(0)] = 0$;
\item[(d)] $\lim_{x\to\infty} Y(x) = 0$ in probability.
\end{itemize}
\end{theorem}
The equivalence of (a), (b), (c) in Theorem~\ref{theo:MMM}   was known before (see Theorem~3.4 in~\cite{S08} for the equivalence of (a) and (c), and Theorem~1.1 in~\cite{KS10} for the equivalence of (a) and (b)). We shall prove in Section~\ref{app:PN} that (c) is equivalent to (d). Moreover, we shall introduce a decomposition of the process $\eta$ into a mixing part and a part containing no mixing components.

Finally, we can characterize the mixed moving maximum property. The definition of this property will be recalled in Section~\ref{app:CD}.
\begin{theorem}\label{theo:MMM}
For a stationary, stochastically continuous max-stable process $\eta$ with locally bounded sample paths, the following conditions are equivalent:
\begin{itemize}
\item[(a)] $\eta$ has a mixed moving maximum representation;
\item[(b)] $\eta$ has no conservative component in its spectral representation, that is $\eta_C=0$;
\item[(c)] $\int_{\mathcal{X}} Y(x) \lambda(\mathrm{d} x) < +\infty$ almost surely;
\item[(d)] $\lim_{x\to\infty} Y(x) = 0$ almost surely.
\end{itemize}
\end{theorem}
The equivalence of (a), (b), (c) in Theorem~\ref{theo:MMM} was known before and holds even without the assumption of local boundedness (see Sections~\ref{subsec:def_CD}, \ref{subsec:3.2} and the references therein). Our main contribution is an alternative characterization of the conservative/dissipative decomposition stated in Proposition~\ref{prop} that implies the equivalence of (c) and (d).  This equivalence may look strange at a first glance because neither (c) implies (d) nor it is implied by (d) for a \emph{general} stochastic process $Y$.    However, the process $Y$ appearing in Theorems~\ref{theo:ergodicity}, \ref{theo:mixing}, \ref{theo:MMM} is subject to the restriction that it leads to a \emph{stationary} process $\eta$.  Processes $Y$ with this property were called Brown--Resnick stationary in~\cite{KSdH09}. Another restriction appearing in Theorem~\ref{theo:MMM} is the local boundedness of $\eta$.  This condition cannot be removed, as will be shown in Example~\ref{eq:loc_bound}.  A special case of the implication (d) $\Rightarrow$ (c) when $\log Y$ is a Gaussian process with stationary increments and certain drift was obtained in~\cite[Theorem~7.1]{WS10}.

The rest of the paper is structured as follows. Section~\ref{sec:preliminaries} is devoted to preliminaries on non-singular ergodic theory and cone decompositions for max-stable processes. Section~\ref{app:CD} reviews known results on the conservative/dissipative decompositions and provides  an alternative definition via a simple cone decomposition with an emphasis on the case of locally bounded max-stable processes.  Section~\ref{app:PN} introduces the positive/null decomposition and proposes an alternative construction via another simple cone decomposition.
In Section~\ref{sec:mixing} we study mixing.


\section{Preliminaries}\label{sec:preliminaries}

\subsection{Non-singular flow representations of max-stable processes}
We recall some information on non-singular flow representations of stationary max-stable processes.
For more details on non-singular ergodic theory, the reader should refer to Krengel~\cite{K85}, Aaronson~\cite{A98} or Danilenko and Silva~\cite{danilenko_silva}.

\begin{definition}
A measurable non-singular flow on a measure space $(S,\cB,\mu)$ is a family of functions $\phi_x:S\to S$, $x\in \mathcal{X}$, satisfying
\begin{itemize}
\item[(i)] (flow property) for all $s\in S$ and $x_1,x_2\in \mathcal{X}$,
\[
\phi_0(s)=s \quad \mbox{and} \quad \phi_{x_1+x_2}(s)=\phi_{x_2}(\phi_{x_1}(s));
\]
\item[(ii)] (measurability) the mapping $(x,s)\mapsto \phi_x(s)$ is measurable from $\mathcal{X}\times S$ to $S$;
\item[(iii)] (non-singularity) for all $x\in \mathcal{X}$, the measures $\mu \circ\phi_x^{-1}$ and $\mu$ are equivalent, i.e.\ for all $A\in\cB$, $\mu(\phi_x^{-1}(A))=0$ if and only if $\mu(A)=0$.
\end{itemize}
\end{definition}
The non-singularity property ensures that one can define the Radon--Nikodym derivative
\begin{equation}\label{eq:RN}
\omega_x(s)=\frac{\mathrm{d}(\mu \circ\phi_x)}{\mathrm{d}\mu}(s).
\end{equation}
By the measurability property, one may assume that the mapping $(x,s)\mapsto \omega_x(s)$ is jointly measurable on $\mathcal{X}\times S$.

According to  de Haan and Pickands~\cite{dHP86}, see also~\cite{K09} and~\cite{WS10}, any stochastically continuous \emph{stationary} max-stable process $\eta$ admits a (distributional) representation of the form
\begin{equation}\label{eq:flow}
\eta(x)=\bigvee_{i\geq 1} U_if_x(s_i),\quad x\in \mathcal{X},
\end{equation}
where $f_x(s)=\omega_x(s)f_0(\phi_x(s))$ and
\begin{itemize}
\item[-] $(\phi_x)_{x\in \mathcal{X}}$ is a measurable non-singular flow on some $\sigma$-finite measure space $(S,\cB,\mu)$, with $\omega_x(s)$ defined by \eqref{eq:RN},
\item[-] $f_0\in L^1(S,\cB,\mu)$ is non-negative
such that the set $\{f_0=0\}$ contains  no $(\phi_x)_{x\in\cX}$--invariant set $B\in \cB$ of positive measure,
\item[-] $\{(s_i,U_i)\}_{i\geq 1}$ is some enumeration of the points of the Poisson point process on $S\times (0,+\infty)$ with intensity $\mu(\mathrm{d}s)\times u^{-2}\mathrm{d}u$.
\end{itemize}

If $(S,\cB,\mu)$ is a \emph{probability} space, the point process $\{(s_i,U_i)\}_{i\geq 1}$ can be generated by taking $(s_i)_{i\geq 1}$ to be i.i.d.\ random elements in $S$ with probability distribution $\mu$, that are independent from $(U_i)_{i\geq 1}$.
Thus, one easily recovers the de Haan representation~\eqref{eq:deHaan} by considering the i.i.d.\ stochastic processes
$Y_i(x)=f_x(s_i)$, $i\geq 1$.

The flow representation \eqref{eq:flow} is comonly written as an extremal integral
\begin{equation}\label{eq:eint}
\eta(x)=\int^{\rm{e}}_{S} f_x(s)M(\mathrm{d}s) ,\quad x\in \mathcal{X},
\end{equation}
where $M(\mathrm{d}s)$ denotes a $1$-Fr\'echet random sup-measure on $(S,\cB)$ with control measure $\mu$. The reader should refer to Stoev and Taqqu \cite{ST05} for more details on extremal integrals.
In the present paper, one can simply view the extremal integral \eqref{eq:eint} as a shorthand for the pointwise maximum over a Poisson point process \eqref{eq:flow}.

\subsection{Cone-based decompositions}
In the spirit of  Wang and Stoev \cite[Theorem 4.2]{WS10} and Dombry and Kabluchko \cite[Lemma 16]{DK15a}, we shall use decompositions of max-stable processes based on cones.
We denote by $\cF_0=\cF( \mathcal{X},[0,+\infty))\setminus\{0\}$ the set of  non-negative measurable functions on $\mathcal{X}$ excluding the zero function.
A subset $\mathcal{C}\subset \mathcal{F}_0$ is called a {\it cone} if for all $f\in\mathcal{C}$ and $u>0$, $uf\in\mathcal{C}$.
The cone $\mathcal{C}$ is said to be {\it shift-invariant} if for all $f\in\mathcal{C}$ and $x\in\mathcal{X}$ we have $f(\cdot+x)\in\mathcal{C}$.
\begin{lemma}[Lemma 16 in~\cite{DK15a}]\label{lem:dec}
Let $\mathcal{C}_1$ and $\mathcal{C}_2$ be two shift-invariant cones such that $\mathcal{F}_0=\mathcal{C}_1\cup\mathcal{C}_2$ and $\mathcal{C}_1\cap\mathcal{C}_2=\varnothing$.
Let $\eta$ be a stationary max-stable process given by representation \eqref{eq:deHaan} such that the events $\{Y_i\in \mathcal{C}_1\}$ and $\{Y_i\in \mathcal{C}_2\}$ are measurable. Consider the decomposition $\eta=\eta_1\vee \eta_2$ with
\[
\eta_1(x)=\bigvee_{i\geq 1} U_iY_i(x)\ind_{\{Y_i\in\mathcal{C}_1\}}\quad \mbox{and}\quad\eta_2(x)=\bigvee_{i\geq 1} U_iY_i(x)\ind_{\{Y_i\in\mathcal{C}_2\}}.
\]
Then, $\eta_1$ and $\eta_2$ are stationary and independent max-stable processes whose  distribution depends only on the distribution of $\eta$ and not on the specific representation \eqref{eq:deHaan}.
\end{lemma}

\section{Conservative/dissipative decomposition} \label{app:CD}
\subsection{Definition of the conservative/dissipative decomposition}\label{subsec:def_CD}
We recall  the Hopf (or conservative/dissipative) decomposition from non-singular ergodic theory; see Aaronson \cite{A98}. We start with the discrete case $\mathcal{X}=\mathbb{Z}^d$.
\begin{definition}
Consider a measure space $(S,\cB,\mu)$ and a non-singular flow $(\phi_x)_{x\in\mathbb{Z}^d}$. A measurable set $W\subset S$ is said to be
{\it wandering} if the sets $\phi_x^{-1}(W)$, $x\in\mathbb{Z}^d$, are disjoint.
\end{definition}
The Hopf decomposition theorem states that there exists a  partition of $S$
into two disjoint measurable sets $S=C\cup D$, $C\cap D=\varnothing$, such that
\begin{itemize}
\item[(i)] $C$ and $D$ are $(\phi_x)_{x\in\mathbb Z^d}$--invariant,
\item[(ii)] there exists no wandering set $W\subset C$ with positive measure,
\item[(iii)] there exists a wandering set $W_0\subset D$ such that $D=\cup_{x\in\mathbb Z^d}\phi_{x}(W_0)$.
\end{itemize}
This decomposition is unique mod $\mu$ and is called the \textit{Hopf decomposition} of $S$ associated with the flow $(\phi_x)_{x\in\mathbb{Z}^d}$; the sets $C$ and $D$ are called the \textit{conservative} and \textit{dissipative} parts respectively.
In the case when $\mathcal{X}=\mathbb{R}^d$, we follow Roy~\cite{R10} by defining the Hopf decomposition of $S$ associated with a measurable flow $(\phi_x)_{x\in\mathbb{R}^d}$ as the Hopf decomposition associated with the discrete skeleton flow $(\phi_x)_{x\in\mathbb{Z}^d}$.


One can then introduce the conservative/dissipative decomposition of the max-stable process $\eta$ given by \eqref{eq:flow}, \eqref{eq:eint}: we have $\eta=\eta_C\vee\eta_D$ with
\begin{equation}
\eta_C(x)=\int^{\rm {e}}_{C} f_x(s)M(\mathrm{d}s)\quad \mbox{and}\quad
\eta_D(x)=\int^{\rm {e}}_{D} f_x(s)M(\mathrm{d}s) ,\quad x\in \mathcal{X}.\label{eq:flow_CD}
\end{equation}
The processes $\eta_C$ and $\eta_D$ are independent and their distribution depends only on the distribution of $\eta$ and not on the particular choice of the representation \eqref{eq:flow}.

The importance of the conservative/dissipative decomposition comes from the notion of mixed moving maximum representation.
\begin{definition}\label{def:MMM}
A  stationary max-stable process $(\eta(x))_{x\in \mathcal X}$ is said to have a mixed moving maximum representation (shortly M3-representation) if
\[
\eta(x)\stackrel{f.d.d.}{=}\bigvee_{i\geq 1} V_i Z_i(x-X_i),\quad x\in \cX,
\]
where
\begin{itemize}
\item[-] $\{(X_i,V_i),i\geq 1\}$ is a Poisson point process on $\cX \times (0,+\infty)$ with intensity $\lambda(\mathrm{d}x)\times u^{-2}\mathrm{d}u$,
\item[-] $(Z_i)_{i\geq 1}$ are i.i.d.\ copies of a non-negative measurable stochastic process $Z$ on $\mathcal{X}$ satisfying $\bbE [\int_{\cX} Z(x)\lambda(\mathrm{d}x)] < +\infty$,
\item[-] $\{(X_i,V_i),i\geq 1\}$ and $(Z_i)_{i\geq 1}$ are independent.
\end{itemize}
\end{definition}
The following important theorem relates the dissipative/conservative decomposition and the existence of an M3-representation; see Wang and Stoev \cite[Theorem 6.4]{WS10} in the max-stable case with $d=1$ or Roy \cite[Theorem 3.4]{R10} in the sum-stable case with $d\geq 1$.
\begin{theorem}\label{theo:M3}
Let $\eta$ be a stationary max-stable process given by the non-singular flow representation \eqref{eq:flow}. Then, $\eta$ has an M3-representation if and only if $\eta$ is generated by a dissipative flow.
\end{theorem}

\subsection{Characterization using spectral functions}\label{subsec:3.2}

The following simple integral test on the spectral functions allows us to retrieve the conservative/dissipative decomposition; see Roy and Samorodnitsky \cite[Proposition]{RS08}, Roy \cite[Proposition 3.2]{R10} and Wang and Stoev \cite[Theorem 6.2]{WS10}.
\begin{theorem}\label{theo:CD} We have
\begin{itemize}
\item[(i)]$\int_{\cX}f_x(s)\lambda(\mathrm{d}x)=\infty$  $\mu(\mathrm{d}s)$--a.e. on $C$;
\item[(ii)]$\int_{\cX}f_x(s)\lambda(\mathrm{d}x)<\infty$ $\mu(\mathrm{d}s)$--a.e. on $D$.
\end{itemize}
\end{theorem}

Consider a stationary  max-stable process $\eta$ given by de Haan's representation~\eqref{eq:deHaan}.
In view of Theorem \ref{theo:CD}, we introduce the cones of functions
\begin{eqnarray}
\cF_C&=&\left\{f\in \cF_0;\ \int_{\cX}f(x)\lambda(\mathrm{d}x)=\infty \right\},\label{eq:CC}\\
\cF_D&=&\left\{f\in \cF_0;\ \int_{\cX}f(x)\lambda(\mathrm{d}x)<\infty  \right\}\label{eq:CD}.
\end{eqnarray}
These cones are clearly shift-invariant and, assuming that $Y$ is jointly measurable and separable, the events $\{Y\in \cF_C\}$ and $\{Y\in \cF_D\}$ are measurable. Using Lemma~\ref{lem:dec}, we define
\begin{equation}\label{eq:cons}
\eta_C(x)= \bigvee_{i\geq 1} U_iY_i(x)\ind_{\{Y_i\in\cF_C\}}\quad \mbox{and}\quad \eta_D(x)= \bigvee_{i\geq 1} U_iY_i(x)\ind_{\{Y_i\in\cF_D\}}.
\end{equation}
Using Theorem \ref{theo:CD} and Lemma \ref{lem:dec} one can easily prove   that we retrieve (in distribution) the conservative/dissipative decomposition \eqref{eq:flow_CD} based on the flow representation~\eqref{eq:flow}.

The main contribution of this section concerns the case when the max-stable process $\eta$ has locally bounded sample paths, which is usually the case in applications. Interestingly, one can then introduce another, more simple and  convenient, cone decomposition equivalent to \eqref{eq:cons}. Consider
\begin{eqnarray*}
\tilde\cF_C&=&\left\{f\in \cF_0;\ \limsup_{x\to\infty} f(x)>0 \right\},\\
\tilde\cF_D&=&\left\{f\in \cF_0;\ \lim_{x\to\infty}f(x)=0 \right\}.
\end{eqnarray*}
Note that since the process $Y$ is assumed to be separable, the events $\{Y\in\tilde\cF_C\}$ and $\{Y\in\tilde\cF_C\}$ are measurable.
\begin{proposition}\label{prop}
Let $\eta$ be a stationary max-stable process given by de Haan's representation \eqref{eq:deHaan} and assume that $\eta$ has locally bounded sample paths.
Then, modulo null sets,
\[
 \{Y\in\cF_C\}= \{Y\in\tilde\cF_C\}\quad \mbox{and}\quad  \{Y\in\cF_D\}= \{Y\in\tilde\cF_D\}.
\]
We deduce that the decomposition
\[
\tilde\eta_C(x)= \bigvee_{i\geq 1} U_iY_i(x)\ind_{\{Y_i\in\tilde\cF_C\}}\quad \mbox{and}\quad
\tilde\eta_D(x)= \bigvee_{i\geq 1} U_iY_i(x)\ind_{\{Y_i\in\tilde\cF_D\}}.
\]
is almost surely equal to the decomposition~\eqref{eq:cons}.
\end{proposition}

\begin{proof}
We consider first the discrete setting $\mathcal{X}=\mathbb{Z}^d$.
The convergence of the series $\sum_{x\in\mathbb{Z}^d} f(x)$ implies the convergence $\lim_{x\to\infty} f(x)=0$
 so that the inclusion $\{Y\in \cF_D\}\subset \{Y\in \tilde\cF_D\}$ is trivial. We need only to
  prove the converse inclusion $\{Y\in \tilde\cF_D\}\subset \{Y\in \cF_D\}$.
   Then, the equality $\{Y\in\cF_D\}=\{Y\in \tilde\cF_D\}$ (modulo null sets) implies the
    equality of the complementary sets, i.e. $\{Y\in \cF_C\}=\{Y\in \tilde\cF_C\}.$

\vspace*{2mm}
\noindent
\textit{Proof of the inclusion $\{Y\in\tilde\cF_D\}\subset \{Y\in \cF_D\}$.}
Let $\tilde Y_D=Y\ind_{\{Y\in\tilde\cF_D\}}$ and $\tilde\eta_D=\vee_{i\geq 1}U_iY_i\ind_{\{Y_i\in\tilde\cF_D\}}$. We shall show that $\tilde\eta_D$ admits an M3-representation. By Theorem \ref{theo:M3}, this implies that $\tilde Y_D$ belongs a.s.\ to $\cF_D$ and hence $\{Y\in\tilde\cF_D\}\subset \{Y\in \cF_D\}$ modulo null sets.
For the sake of notational convenience,  we assume that $Y\in\tilde\cF_D$ a.s.\ so that $\tilde Y_D=Y$ and $\tilde \eta_D=\eta$. We prove that $\eta$ has an M3-representation with a strategy similar to the proof of Theorem 14 in Kabluchko {\it et al.}\ \cite{KSdH09}. We sketch only the main lines.
We  introduce the random variables
\begin{equation}\label{eq:techMMM2}
 X_i=\mathop{\mathrm{argmax}}_{x\in\cX}Y_i(x),\quad Z_i(\cdot)=\frac{Y_i(X_i+\cdot)}{\max_{x\in\cX} Y_i (x)},\quad V_i=U_i\max_{x\in\cX} Y_i (x).
\end{equation}
If the $\mathop{\mathrm{argmax}}$ is not unique, we use the lexicographically smallest value. Clearly, we have $U_iY_i(x)=V_iZ_i(x-X_i)$ for all $x\in\cX$ so that
\[
\eta(x)=\bigvee_{i\geq 1}V_iZ_i(x-X_i).
\]
It remains to check that $(X_i,V_i,Z_i)_{i\geq 1}$ has the  properties required in Definition~\ref{def:MMM}, i.e.\ is a Poisson point process on $\mathcal X \times (0,\infty)\times \mathcal F_0$ with  intensity measure
$\lambda(\mathrm{d}x) \times u^{-2}\mathrm{d}u \times Q(\mathrm{d}f)$, where $Q$ is a probability measure on $\cF_0$.
 Clearly, $(X_i,V_i, Z_i)_{i\geq 1}$ is a Poisson point process as the image of the original point process $(U_i,Y_i)_{i\geq 1}$. Its intensity  is the  image of the  intensity of the original point process.
With a straightforward transposition of the arguments of \cite[Theorem 14]{KSdH09}, one can check that it has the required form.

\vspace*{2mm}
We now turn to the case $\mathcal{X}=\mathbb{R}^d$.
The convergence of the integral $\int_{\cX} f(x) \lambda(\mathrm{d}x)$ does not imply the convergence $\lim_{x\to\infty} f(x)=0$. But it is easy to prove that for $K=[-1/2,1/2]^d$, the convergence  of the integral $\int_{\cX} \sup_{u\in K}f(x+u) \lambda(\mathrm{d}x)$ implies the convergence $\lim_{x\to\infty} f(x)=0$.
We introduce the cone
\[
\cF_D'= \left\{f\in\cF_0;\ \int_{\cX} \sup_{u\in K}f(x+u) \lambda(\mathrm{d}x) <\infty\right\}.
\]
The inclusions of cones
$\cF_D'\subset \cF_D$ and $\cF_D'\subset \tilde\cF_D$
imply the trivial inclusions of events
\[
\{Y\in \cF_D'\}\subset \{Y\in\cF_D\}\quad \mbox{and}\quad \{Y\in\cF_D'\}\subset \{Y\in\tilde\cF_D\}.
\]
We shall prove below that, modulo null sets,
\[
\{Y\in\cF_D\}\subset \{Y\in\cF_D'\} \quad \mbox{and}\quad \{Y\in\tilde\cF_D\}\subset \{Y\in \cF_D\}
\]
whence we deduce the equalities, modulo null sets,
\[
\{Y\in\cF_D\}=\{Y\in\cF_D'\}=\{Y\in\tilde\cF_D\},
\]
proving the proposition.

\vspace*{2mm}
\noindent
\textit{Proof of the inclusion $\{Y\in\cF_D\} \subset \{Y\in \cF_D'\}$.}
Let $Y_D=Y\ind_{\{Y\in\cF_D\}}$ and $\eta_D=\vee_{i\geq 1}U_iY_i\ind_{\{Y_i\in\cF_D\}}$ be the dissipative part of $\eta$. Theorem~\ref{theo:M3} implies that $\eta_D$ has an M3-representation of the form
\[
\eta_D(x)\stackrel{f.d.d.}=\bigvee_{i\geq 1}V_i Z_{D,i}(x-X_i),\quad x\in\cX.
\]
The fact that $\eta$ is locally bounded implies that $\eta_D$ is a.s.\ finite on $K$ and
\begin{equation}\label{eq:theta_K}
\bbP\left[\sup_{x\in K}\eta_D(x)\leq z\right]= \exp\left(-\frac{\theta_D(K)}{z}\right)
\end{equation}
with
\[
\theta_D(K)=\bbE\left[\int_{\cX} \sup_{x\in K} Z_D(x-y)\lambda(\mathrm{d}y)\right]<\infty.
\]
We deduce that $\int_{\cX} \sup_{x\in K} Z_D(x-y)\lambda(\mathrm{d}y)$ is a.s.\ finite and hence, $Z_D$ belongs a.s.\ to the cone $\cF_D'$. This implies that $Y \ind_{\{Y\in \cF_D\}}\in\mathcal{\cF_D'}$ almost surely, whence $\{Y\in\cF_D\}\subset\{Y\in\cF_D'\}$ modulo null sets.

\vspace*{2mm}
\noindent
\textit{Proof of the inclusion $\{Y\in\tilde\cF_D\}\subset \{Y\in \cF_D\}$.}
With the same notation as in the dicrete case, we show that $\tilde\eta_D$ is generated by a dissipative flow and hence has an M3-representation. By Theorem \ref{theo:M3}, this implies that $\tilde Y_D$ belongs a.s.\ to $\cF_D$ and proves the inclusion $\{Y\in\tilde\cF_D\}\subset \{Y\in \cF_D\}$. Note that the discrete skeleton $\tilde Y_D^{skel}=(\tilde Y_D(x))_{x\in\bbZ^d}$ satisfies $\lim_{x\to\infty} \tilde Y_D^{skel}=0$. We deduce $\tilde Y_D^{skel}\in  \tilde \cF_D$ a.s.\ which is equivalent to
$\tilde Y_D^{skel}\in \cF_D$ a.s.\ (see the proof above in the discrete case). Hence $(\tilde \eta_D(x))_{x\in\bbZ^d}$ is generated by a dissipative flow and this implies that $(\tilde \eta_D(x))_{x\in\bbR^d}$ is generated by a dissipative flow (see \cite[Section 2]{R10}).
\end{proof}

\begin{proof}[Proof of Theorem~\ref{theo:MMM}]
The equivalence of (a), (b), (c) in Theorem~\ref{theo:MMM} was known before and holds even without the assumption of local boundedness (see Section~\ref{subsec:def_CD} and the reference therein). The  equivalence of (c) and (d) holds under the assumption of local boundedness and is a straightforward consequence of Proposition~\ref{prop}.
\end{proof}

\begin{example}\label{eq:loc_bound}{\rm
The assumption that the sample paths of $\eta$ should be locally bounded cannot be removed from Proposition~\ref{prop}. To see this, consider the following (deterministic) process $Z$:
$$
Z(x) = \sum_{n=1}^{\infty} f(n^2(x-n)), \quad x\in\bbR,
$$
where $f(t)=(1-t^2)\ind_{|t|\leq 1}$. 
The process $Z$ is non-zero only on the intervals of the form $(n-\frac 1{n^2}, n+\frac 1{n^2})$, $n\in\mathbb{N}$.  Its sample paths are continuous and bounded on $\mathbb R$.
The M3-process $\eta$ corresponding to $Z$ is well-defined because $\int_{\mathbb{R}}Z(x) {\rm{d}} x <\infty$. On the other hand, $\bbP[Z\in \tilde \cF_D] = 0$ and hence, $\bbP[Y\in \tilde \cF_D] = 0$,
where $Y$ is the spectral function of $\eta$ from the de Haan representation~\eqref{eq:deHaan}.
It is easy to check that
$$
\bbP\left[\sup_{x\in [0,1]}\eta(x)\leq z\right]= \exp\left(-\frac{\theta_{[0,1]}}{z}\right),\quad z>0,
$$
with
$$
\theta_{[0,1]}= \int_{\mathbb{R}} \left(\sup_{x\in [0,1]} Z(x-y)\right) \mathrm{d}y = +\infty,
$$
whence $\sup_{x\in [0,1]}\eta(x)=+\infty$ a.s.\ and the sample paths of $\eta$ are not locally bounded.
}
\end{example}

\section{Positive/null  decomposition}\label{app:PN}
\subsection{Definition of the positive/null decomposition}
We start by defining the Neveu decomposition of the non-singular flow $(\phi_x)_{x\in\cX}$; see, e.g.,\ Krengel \cite[Theorem 3.9]{K85}, Samorodnitsky~\cite{S05} or Wang {\it et al.}\ \cite[Theorem 2.4]{WRS13}.
\begin{definition}
Consider a measure space $(S,\cB,\mu)$ and a measurable non-singular flow $(\phi_x)_{x\in\cX}$ on $S$. A measurable set $W\subset S$ is said to be
{\it weakly wandering} with respect to  $(\phi_x)_{x\in\cX}$ if there exists a sequence  $\{x_n\}_{n\in\bbN}\subset\cX$ such that $\phi_{x_n}^{-1}(W)\cap\phi_{x_m}^{-1}(W)=\varnothing$ for all $n\neq m$.
 \end{definition}
The Neveu decomposition theorem states that there exists a  partition of $S$ into two disjoint measurable sets $S=P\cup N$, $P\cap N=\varnothing$, such that
\begin{itemize}
\item[(i)] $P$ and $N$ are $(\phi_x)_{x\in\cX}$--invariant for all $x\in\cX$,
\item[(ii)] $P$ has no weakly wandering set of positive measure,
\item[(iii)] $N$ is a union of countably many weakly wandering sets.
\end{itemize}
This decomposition is unique mod $\mu$ and is called the \textit{Neveu decomposition} of $S$ associated with $(\phi_x)_{x\in\cX}$; $P$ and $N$ are called the \textit{positive} and \textit{null} components
with respect to  $(\phi_x)_{x\in\cX}$, respectively. It can be shown that $P$ is the largest subset of $S$ supporting a finite measure which is equivalent to $\mu$ and invariant
under the flow $(\phi_x)_{x\in\cX}$ (\cite[Lemma 2.2]{WRS13}). Hence, there exists a finite measure which is equivalent to $\mu$ and invariant under the flow if and only if
$N=\varnothing$ mod $\mu$.

The corresponding positive/null decomposition of the stationary max-stable process $\eta$ represented as in~\eqref{eq:flow}, \eqref{eq:eint} is given by $\eta=\eta_P\vee\eta_N$ with
\begin{equation}\label{eq:posnul}
\eta_P(x)=\int^{\rm {e}}_{P} f_x(s)M(\mathrm{d}s)\quad \mbox{and}\quad
\eta_N(x)=\int^{\rm {e}}_{N} f_x(s)M(\mathrm{d}s) ,\quad x\in \mathcal{X}.\end{equation}
The positive and null components $\eta_P$ and $\eta_N$ are independent, stationary max-stable processes,  and their distribution does not depend on the particular choice of the representation \eqref{eq:flow}.


\subsection{Characterization using spectral functions}
An integral test on the spectral functions which allows to retrieve the positive/null decomposition is known in the one-dimensional case (see Samorodnitsky \cite{S05} or Wang and Stoev \cite[Theorem 5.3]{WS10}).
\begin{theorem}\label{theo:null_positive_old}
Consider the case $d=1$ and introduce the class $\cW$ of positive weight functions $w:\mathcal{X}\to (0,+\infty)$ such that
$\int_{\cX}w(x)\lambda(\mathrm{d}x)<\infty$ and  $w(x)$ and $w(-x)$ are non-decreasing on $\cX\cap [0,+\infty)$.  Then we have
\begin{enumerate}
 \item[(i)] For all $w\in\mathcal{W}$, $\int_{\cX}f_x(s)w(x)\lambda(\mathrm{d}x)=\infty$  $\mu(\mathrm{d}s)$--a.e. on $P$;
\item[(ii)] For some $w\in\mathcal{W}$, $\int_{\cX}f_x(s)w(x)\lambda(\mathrm{d}x)<\infty$  $\mu(\mathrm{d}s)$--a.e. on $N$.
\end{enumerate}
\end{theorem}
The next theorem is a new integral test characterizing the positive/null decomposition.  This test is simpler than Theorem~\ref{theo:null_positive_old} and is valid for all $d\geq 1$. Recall that we write $B_r=[-r,r]^d\cap \mathcal{X}$ for $r>0$.  In the next theorem and its corollary we do not require the sample paths of $\eta$ to be locally bounded.
\begin{theorem}\label{theo:PN}
Let $\eta$ be a stationary, stochastically continuous max-stable process given by the non-singular flow representation~\eqref{eq:flow}. We have
\begin{itemize}
\item[(i)]$\lim_{r\to\infty} \frac{1}{\lambda(B_r)}\int_{B_r}f_x(s)\lambda(\mathrm{d}x)$ exists and is positive  $\mu(\mathrm{d}s)$--a.e.\ on $P$;
\item[(ii)]$\liminf_{r\to\infty} \frac{1}{\lambda(B_r)}\int_{B_r}f_x(s)\lambda(\mathrm{d}x)=0$ $\mu(\mathrm{d}s)$--a.e.\ on $N$.
\end{itemize}
\end{theorem}
\begin{proof}
We consider the positive case and the null case separately.

\vspace*{2mm}
\noindent
\textit{Case 1.}
Assume first that $\eta$ is generated by a positive flow. Then, there is a probability measure $\mu^*$ on $(S, \mathcal{B})$ which is equivalent to $\mu$ and which is invariant under the flow. Note that any property holds $\mu$--a.e.\ if and only if it holds $\mu^*$--a.e. We denote by $D(s) = \frac{\rm{d}\mu}{\rm{d}\mu*}(s)\in (0,\infty)$ the Radon--Nikodym derivative and observe that for every $x\in\mathcal{X}$, the function $f_x^*(s) := f_x(s)D(s)$ satisfies
\begin{equation}\label{eq:f_x_star_invar}
f_x^*(s)=f_0^*(\phi_x(s)) \quad \text{for $\lambda \times \mu$--a.e. $(x,s)\in \mathcal{X}\times S$}.
\end{equation}
Indeed, by definition of $f_x^*$ and $\omega_x$, we have
$$
f_x^*(s)= D(s) f_x(s) = D(s) \omega_x(s) f_0(\phi_x(s)) = \frac{D(s) \omega_x(s)}{D(\phi_x(s))}   f_0^*(\phi_x(s)).
$$
However, recalling the definition \eqref{eq:RN} of $\omega_x(s)$ and that $D(s) = \frac{\rm{d}\mu}{\rm{d}\mu*}(s)\in (0,\infty)$, we obtain
$$
\frac{D(s) \omega_x(s)}{D(\phi_x(s))} = \frac{{\rm{d}}\mu}{{\rm{d}}\mu^*}(s)  \frac{{\rm{d}}(\mu\circ \phi_x)}{{\rm{d}}\mu}(s) \frac{{\rm{d}}(\mu^*\circ \phi_x)}{{\rm{d}}(\mu\circ \phi_x)}(s) = \frac{{\rm{d}}(\mu^*\circ \phi_x)}{{\rm{d}}\mu^*}(s)=1
$$
$\mu$--a.e.\ for every $x\in\mathcal{X}$ because the measure $\mu^*$ is invariant. This yields~\eqref{eq:f_x_star_invar}.
By the multiparameter Birkhoff Theorem (see~\cite[Theorem 2.8]{WRS13}), we have
\begin{equation}\label{eq:birkhoff}
\lim_{r\to\infty} \frac{1}{\lambda(B_r)}\int_{B_r}f_x^*(s)\lambda(\mathrm{d}x)=\bbE[f_0^*|\cI] \quad \text{$\mu^*$--a.e.},
\end{equation}
where $\cI$ is the $\sigma$-algebra of $(\phi_x)_{x\in\cX}$--invariant measurable sets and $\bbE$ denotes the expectation w.r.t.\ $\mu^*$. We prove that the conditional expectation on the right-hand side is a.e.\ strictly positive.  The set $B=\{\bbE[f_0^*|\cI]=0\}$ is measurable and $(\phi_x)_{x\in\cX}$--invariant. Moreover, $f_0^*$ (and hence, $f_0$) vanishes a.e.\ on $B$ since $f_0^*$ is non-negative. This implies that $\mu(B)=0$ by the second condition in the definition of the flow representation~\eqref{eq:flow}. Thus, $\bbE[f_0^*|\cI]>0$ a.e. It follows from~\eqref{eq:birkhoff} and the above considerations  that
\begin{equation}\label{eq:birkhoff1}
\lim_{r\to\infty} \frac{1}{\lambda(B_r)}\int_{B_r}f_x(s)\lambda(\mathrm{d}x)=\frac {\bbE[f_0^*|\cI]} {D(s)} >0 \quad \text{$\mu$--a.e.},
\end{equation}
which proves part (i) of the theorem.


\vspace*{2mm}
\noindent
\textit{Case 2.}
We consider now the case when $\eta$ is generated by a null flow. Let $\mu^*$ be any probability measure on $(S, \mathcal{B})$ which is equivalent to $\mu$. Write $D(s) = \frac{\rm{d}\mu}{\rm{d}\mu*}(s)\in (0,\infty)$ for the Radon--Nikodym derivative. The functions $f_x^*(s) := f_x(s)D(s)$ satisfy
$$
f_x^*(s) = \omega^*_x(s) f_0^*(\phi_x(s)), \quad \text{where } \omega^*_x(s) := \frac{{\rm{d}}(\mu^*\circ \phi_x)}{\rm{d}\mu^*}(s),
$$
by the same considerations as in the positive case.
Birkhoff's ergodic theorem is valid for measure preserving flows only, but we can use Krengel's stochastic ergodic theorem for non-singular actions (see~\cite[Theorem 2.7]{WRS13}) which yields
\[
\frac{1}{\lambda(B_r)}\int_{B_r}f_x^*(\cdot)\lambda(\mathrm{d}x)\stackrel{\mu^*}{\rightarrow} F(\cdot)\quad \mbox{as } r\to\infty
\]
where $\stackrel{\mu^*}{\rightarrow}$ denotes convergence in $\mu^*$-probability and the limit function $F\in L^1(S,\mu^*)$ is such that for all $x\in\cX$,
\[
\omega_x^*(s) F(\phi_x(s))=F(s) \quad \text{a.e.}
\]
This relation implies that the measure $F(s)\mu^*(\mathrm{d}s)$ is a finite measure which is absolutely continuous with respect to $\mu$ and invariant under the flow $(\phi_x)_{x\in\cX}$. Since the flow has no positive component, this means that $F=0$ a.e. We deduce that $\frac{1}{\lambda(B_r)}\int_{B_r}f_x^*(\cdot)\lambda(\mathrm{d}x)$ converges in $\mu^*$-probability to $0$. Convergence in probability implies a.s.\ convergence along a subsequence, whence
\[
\liminf_{r\to\infty}\frac{1}{\lambda(B_{r})}\int_{B_{r}}f_x^*(s)\lambda(\mathrm{d}x)=0 \quad \text{$\mu^*$--a.e.}
\]
Since $f_x$ differs from $f_x^*$ by a positive factor and the measures $\mu$ and $\mu^*$ are equivalent, we have
\[
\liminf_{r\to\infty}\frac{1}{\lambda(B_{r})}\int_{B_{r}}f_x(s)\lambda(\mathrm{d}x)=0 \quad \text{$\mu$--a.e.},
\]
which proves part (ii) of the theorem.
\end{proof}

As a consequence of Theorem~\ref{theo:PN}, we can provide a new construction for the positive/null decomposition \eqref{eq:posnul}.
Consider the following shift-invariant cones
\begin{eqnarray}
\cF_P&=&\left\{f\in \cF_0;\ \lim_{r\to\infty}\frac{1}{\lambda(B_r)}\int_{B_r} f(x)\lambda(\mathrm{d}x)>0\right\},\label{eq:FP}\\
\cF_N&=&\left\{f\in \cF_0;\ \liminf_{r\to\infty}\frac{1}{\lambda(B_r)}\int_{B_r} f(x)\lambda(\mathrm{d}x)=0\right\}.\label{eq:FN}
\end{eqnarray}
In the definition of $\cF_P$ the limit is required to exist and to be positive.
\begin{corollary}\label{cor}
Let $\eta$ be a stationary, stochastically continuous max-stable process given by de Haan's representation \eqref{eq:deHaan}.
Then the decomposition $\eta=\eta_P\vee\eta_N$ with
\[
\eta_P(x)= \bigvee_{i\geq 1} U_iY_i(x)\ind_{\{Y_i\in\cF_P\}} \quad\mbox{and}\quad
\eta_N(x)= \bigvee_{i\geq 1} U_iY_i(x)\ind_{\{Y_i\in\cF_N\}} \label{eq:null2}
\]
is equal (in distribution) to the positive/null decomposition \eqref{eq:posnul}.
\end{corollary}

\begin{proof}
Corollary~\ref{cor} is a direct consequence of Theorem~\ref{theo:PN} and Lemma~\ref{lem:dec}.
Note that although instead of $\cF_P\cup\cF_N=\cF_0$ it holds only that $\mathbb{P}[Y\in \cF_P\cup\cF_N]=1$,  Lemma~\ref{lem:dec} still applies.
\end{proof}

\begin{proof}[Proof of Theorem~\ref{theo:ergodicity}]
We need to prove the equivalence of (c), (e), (f) only; see Section~\ref{subsec:ergodic_properties} for references to the other equivalences. We recall that (c) states that $\eta$ has no positive recurrent component, and
\begin{itemize}
\item[(e)] $\lim_{r\to\infty} \frac 1{\lambda(B_r)} \int_{B_r} Y(x) \lambda(\mathrm{d} x) = 0$ in probability;
\item[(f)] $\liminf_{r\to\infty} \frac 1{\lambda(B_r)} \int_{B_r} Y(x) \lambda(\mathrm{d} x) =0$ a.s.
\end{itemize}
The equivalence of (c) and (f) follows from Corollary~\ref{cor}. Clearly, (e) implies (f) because any sequence converging to $0$ in probability has a subsequence converging to $0$ a.s.

It remains to show that (c) implies (e). Since the positive/null decomposition of $\eta$ does not depend on the choice of the flow representation, we can consider a \textit{minimal} representation $(f_x)_{x\in\mathcal X}$ of $\eta$ by a null-recurrent flow $(\phi_x)_{x\in\mathcal X}$ on a probability space $(S^*,\mathcal B^*, \mu^*)$; see~\cite[Section~3]{WS10} for definition and existence of the minimal representation.   In the proof of Theorem~\ref{theo:PN}, Case~2, we have shown that
$$
M_r := \frac{1}{\lambda(B_r)}\int_{B_r}f_x\lambda(\mathrm{d}x) \underset{r\to\infty}{\longrightarrow} 0 \quad \text{in probability on $(S^*,\mathcal B^*, \mu^*)$}.
$$
However, we are interested in an arbitrary de Haan representation $(Y(x))_{x\in\mathcal X}$ of $\eta$ on a probability space $(S, \mathcal B, \mu)$. This representation need not be generated by a flow, but it can be mapped to the minimal one (see~\cite[Theorem~3.2]{WS10}). More concretely, there is a measurable map $\Phi:S\to S^*$ and a measurable function $h:S\to (0,\infty)$ such that for every $x\in\mathcal X$,
$$
Y(x; s) = h(s) f_x(\Phi(s)) \quad \text{for $\mu$-a.e.\ $s\in S$},
$$
and $\mu^*$ is the push-forward of the (probability) measure $\mu_h({\rm{d}} s) := h(s) \mu({\rm{d}} s)$ by the map $\Phi$. We have
$$
\frac 1 {\lambda(B_r)} \int_{B_r}Y(x;s)\lambda({\rm{d}}x) =
h(s)\cdot M_r(\Phi(s))  \quad \text{for $\mu$-a.e.\ $s\in S$}.
$$
Since $M_r\to 0$ in $\mu^*$-probability as $r\to\infty$, we obtain that for every $\eps>0$,
$$
\mu_h\{M_r\circ \Phi >\eps\} = (\mu_h \circ \Phi^{-1}) \{M_r>\eps\} = \mu^*\{M_r>\eps\}  \underset{r\to\infty}{\longrightarrow} 0.
$$
Since $h$ is strictly positive, this implies that $\mu\{M_r\circ \Phi >\eps\}\to 0$ and hence, $h\cdot (M_r\circ \Phi) \to 0$ in $\mu$-probability, thus proving (e).
\end{proof}


\section{Mixing}\label{sec:mixing}

\subsection{Proof of Theorem~\ref{theo:mixing}}
We need to prove the equivalence of (c) and (d) only, that is
$$
\text{(c):}\; \lim_{x\to\infty} \mathbb E [Y(x) \wedge Y(0)] = 0
\quad \Leftrightarrow \quad
\text{(d):} \; \lim_{x\to\infty} Y(x) = 0 \text{ in probability}.
$$
See Section~\ref{subsec:ergodic_properties} for references to the other equivalences.

Assume that (d) holds, i.e.\ $\lim_{x\to\infty} Y(x)=0$ in probability.  The upper bound $Y(x)\wedge Y(0)\leq Y(0)$ with $Y(0)$ integrable implies that the collection $(Y(x)\wedge Y(0))_{x\in\mathcal{X}}$ is uniformly integrable. Assumption (d) implies that  $Y(x)\wedge Y(0)$ converges in probability to $0$ as $x\to\infty$, whence we deduce that $\mathbb{E}[Y(x)\wedge Y(0)]\to 0$ as $x\to \infty$, i.e.\ (c) is satisfied.

Conversely, we prove the implication (c) $\Rightarrow$ (d).  We may assume that the scale parameter of $\eta(x)$ is $1$, that is $\mathbb P[\eta(x)\leq u] = {\rm e}^{-1/u}$, $u\geq 0$, and $\mathbb E[Y(x)] = 1$, $x\in\mathcal X$.  The relation
$$
\bbE[Y(x)\wedge Y(0)]=2+\log\bbP[\eta(x)\leq 1,\eta(0)\leq 1]
$$
together with the stationarity of $\eta$ implies that for all $x_0\in\cX$,
\begin{equation}\label{eq:E_min_Y_Y_0}
\lim_{x\to\infty}\bbE[Y(x)\wedge Y(x_0)]=0.
\end{equation}
Without restriction of generality we can assume that $\bbP[Y\equiv 0] = 0$ (where, by separability, the event $\{Y\equiv 0\}$ is interpreted as $\cap_{x\in T}\{Y(x)=0\}$ with countable $T\subset \mathcal X$).  Then,  for arbitrary $\varepsilon>0$, there exists $\alpha>0$ and $x_1,\ldots,x_k\in\cX$ such that $\bbP[ \cup_{1\leq i\leq k} \{ Y(x_i)>\alpha\} ]\geq 1-\varepsilon/2$, whence
\[
\bbP[ Y(x_1)+\ldots+Y(x_k)>\alpha ]\geq 1-\varepsilon/2.
\]
With the inequality $(a_1+\ldots+a_k)\wedge b\leq a_1\wedge b+\ldots+a_k\wedge b$, we obtain from~\eqref{eq:E_min_Y_Y_0} that
\[
\lim_{x\to\infty}\bbE[Y(x)\wedge (Y(x_1)+\ldots+Y(x_k))]=0.
\]
These two equations imply,  for all $\delta>0$,
\begin{align*}
\bbP[Y(x)>\delta]
&\leq \bbP[Y(x)>\delta,\ Y(x_1)+\ldots+Y(x_k)>\alpha]+\varepsilon/2\\
&\leq \bbP[Y(x) \wedge (Y(x_1)+\ldots+Y(x_k)) >\delta\wedge\alpha]+\varepsilon/2\\
&\leq \bbE[Y(x)\wedge (Y(x_1)+\ldots+Y(x_k))]/(\delta\wedge\alpha)+\varepsilon/2\\
&\leq \varepsilon
\end{align*}
for large $|x|$.
This proves that $Y(x)\to 0$ in probability as $x\to\infty$.

\subsection{Criterium for mixing in terms of flows}
Given a measurable non-singular flow $(\phi_x)_{x\in \mathcal{X}}$ on a $\sigma$-finite measure space $(S,\cB,\mu)$ define the corresponding group of $L^1$--isometries $(U_x)_{x\in \mathcal{X}}$ by
$$
(U_x g)(s) = \omega_x(s) g(\phi_x(s)), \quad  g\in L^1(S,\mu), \quad x\in \mathcal{X},
$$
where $\omega_x$ is the Radon--Nikodym derivative; see~\eqref{eq:RN}.
\begin{theorem}\label{theo:mixing_for_flows}
Let $\eta$ be a stationary, stochastically continuous max-stable process with a flow representation~\eqref{eq:flow}. Then, the following conditions are equivalent:
\begin{itemize}
\item[(a)] $\eta$ is mixing.
\item[(b)] $\lim_{x\to\infty} \int_S (f_x\wedge f_0){\rm{d}} \mu = 0$.
\item[(c)] $f_x\to 0$ locally in measure as $x\to\infty$. That is, for every  measurable set $B\subset S$ with $\mu(B)<\infty$ and every $\eps>0$ we have
$$
\lim_{x\to\infty} \mu(B \cap \{f_x>\eps\}) = 0.
$$
\item[(d)] For every non-negative function $g\in L^1(S,\mu)$ we have
$$
\lim_{x\to\infty} \int_S ( (U_x g) \wedge g) {\rm{d}} \mu = 0.
$$
\item[(e)] For every non-negative function $g\in L^1(S,\mu)$, $U_x g \to 0$  locally in measure.
\end{itemize}
\end{theorem}

\begin{proof}
The equivalence of (a) and (b) is due to Stoev; see Theorem~3.4 in~\cite{S08}. We prove that (b) is equivalent to (c), (d), (e).

Take a non-negative function $g\in L^1(S,\mu)$. We prove that the following conditions are equivalent:
\begin{itemize}
\item [(b')] $\lim_{x\to\infty} \int_S ((U_x g)\wedge g){\rm{d}} \mu = 0$.
\item [(c')] $U_x g\to 0$ locally in measure, as $x\to\infty$.
\end{itemize}
Once the equivalence of (b') and (c') has been established, we immediately obtain the equivalence of (b) and (c) (by taking $g=f_0$) and the equivalence of (d) and (e).

\vspace*{2mm}
\noindent
\textit{Proof of (c') $\Rightarrow$ (b').}
Let $U_x g\to 0$ locally in measure, as $x\to\infty$. We prove that (b') holds. Fix some $\eps>0$. The sets $B_n:=\{g>\frac 1n\}$, $n\in\mathbb N$, are measurable, have finite measure (since $g\in L^1(S,\mu)$), and
$$
\lim_{n\to\infty}\int_S g \ind_{S\bsl B_n} {\rm{d}}\mu=0
$$
by the dominated convergence theorem. Hence, by taking $n$ sufficiently large we can achieve that the set $B=B_n$ satisfies $\mu(B) < \infty$ and
$$
\int_{S\bsl B} g {\rm{d}}\mu \leq \eps.
$$
The collection $(U_x g \wedge g)_{x\in\mathcal X}$ is uniformly integrable on $B$ since $U_x g \wedge g\leq g$. Also, we know that $U_x g \wedge g \to 0$ (as $x\to\infty$) in measure on $B$. It follows that
$$
\lim_{x\to\infty} \int_B U_x g \wedge g \mathrm{d} x = 0.
$$
Thus, condition (b') holds.

\vspace*{2mm}
\noindent
\textit{Proof of (b') $\Rightarrow$ (c').} We argue by contradiction. Assume that $U_x g \nrightarrow 0$ locally in measure as $x\to\infty$. Our aim is to prove that $(b')$ is violated. By our assumption, there is a measurable set $B\subset S$ and $\eps>0$ such that $0 < \mu(B) < \infty$ and
\begin{equation}\label{eq:proof_mix5}
\mu(\{U_{x_i} g > \eps\}\cap B) > \eps, \quad i\in \mathbb N,
\end{equation}
where $x_1,x_2,\ldots \to \infty$ is some sequence in $\mathcal{X}$. Denote by $\mathcal{H}$ the family consisting of the sets $\supp U_x g$, $x\in \mathcal{X}$, together with all measurable subsets of these sets. Let $S^*$ be the measurable union of this family; see~\cite[pp.~7--8]{A98} for the proof of its existence. By the exhaustion lemma~\cite[pp.~7--8]{A98}, we can find countably many sets $A_1,A_2,\ldots\in \mathcal{H}$ such that $S^*=A_1\cup A_2\cup\ldots$. It follows that we can find finitely many $z_1,\ldots,z_m\in\mathcal{X}$ such that
$$
\mu\left( (B\cap S^*) \bsl \bigcup_{j=1}^m \supp U_{z_j}g\right) < \frac{\eps}{2}.
$$
Together with~\eqref{eq:proof_mix5} (where $B$ can be replaced by $B\cap S^*$ because $\{U_{x_i} g > \eps\}\subset S^*\mod \mu$),
this implies that  for all $i\in\mathbb N$,
$$
\mu\left(\{U_{x_i}g > \eps\} \cap \bigcup_{j=1}^m \supp U_{z_j}g\right) > \frac{\eps}{2}.
$$
It follows that there is $j\in\{1,\ldots,m\}$ and a subsequence $y_1,y_2,\ldots\to \infty$ of $x_1,x_2,\ldots$ such that for all $i\in \mathbb {N}$,
$$
\mu\left(\{U_{y_i}g > \eps\} \cap \supp U_{z_j}g\right) > \frac{\eps}{2m}.
$$
Put $z=z_j$. For a sufficiently small $\delta\in (0,\eps)$  we have
\begin{equation}\label{eq:proof_mix4}
\mu\left(\{U_{y_i}g > \delta\} \cap \{U_{z}g > \delta\}\right) > \frac{\eps}{4m}.
\end{equation}
By the flow property and~\eqref{eq:proof_mix4} it follows that for all $i\in \mathbb{N}$,
$$
\int_{S} ((U_{y_i-z}g) \wedge g) {\rm{d}}\mu = \int_{S} ((U_{y_i}g) \wedge (U_zg) ) {\rm{d}}\mu >  \frac{\eps}{4m} \delta > 0.
$$
But this contradicts (b').


\vspace*{2mm}
\noindent
\textit{Proof of (d) $\Rightarrow$ (b).} Trivial, because $f_x = U_x f_0$.

\vspace*{2mm}
\noindent
\textit{Proof of (b) $\Rightarrow$ (d).} 
For every non-negative function $g\in L^1(S,\mu)$ we have to show that
$$
\lim_{x\to\infty} \int_S (U_x g\wedge g) {\rm{d}}\mu = 0.
$$
Fix some $\eps>0$. By the same argument relying on the dominated convergence theorem as above, we can find a sufficiently large $K>0$ such that the set $B:= \{1/K \leq g \leq K\}$ satisfies
\begin{equation}\label{eq:proof_mix_bd1}
\int_{S\bsl B} g {\rm{d}\mu} < \eps.
\end{equation}
The set $B$ has finite measure because $g$ is integrable. By the uniform integrability of a single function $g$, there is $\delta>0$ such that every for every measurable set $A\subset B$ with $\mu(A)<\delta$ we have $\int_A g {\rm{d}\mu} < \eps$.

We argue that it is possible to find finitely many $z_1,\ldots,z_m\in \mathcal{X}$ such that the sets $\supp f_{z_1}, \ldots, \supp f_{z_m}$ cover $B$ up to a set of measure at most $\delta/2$. Indeed, let $\mathcal{H}$ be the family consisting of the sets $\supp f_x$, $x\in \mathcal{X}$, together with all measurable subsets of these sets. In the definition of the flow representation~\eqref{eq:flow} we made a ``full support'' assumption which assures that the measurable union of $\mathcal{H}$ is the whole of $S$. By the exhaustion lemma~\cite[pp.~7--8]{A98}, we can represent $S$ as a disjoint union of countably many sets $A_1,A_2,\ldots \in \mathcal{H}$. It follows that we can find finitely many $z_1,\ldots,z_m\in \mathcal{X}$ such that
$$
\mu\left(B \bsl \bigcup_{j=1}^m \supp f_{z_j}\right) <\frac{\delta}{2}.
$$
By taking $c>0$ sufficiently small, we can even achieve that the sets $\{f_{z_1}>c\},\ldots,\{f_{z_m}>c\}$ cover $B$ up to a set of measure at most $\delta$, that is for
$$
D := B \bsl \bigcup_{j=1}^m \{f_{z_j}>c\}
$$
we have $\mu(D)<\delta$. By construction of $\delta$ it follows that
\begin{equation}\label{eq:proof_mix_bd2}
\int_D g {\rm{d}\mu} < \eps.
\end{equation}
For every $j\in\{1,\ldots,m\}$, on the set $A_j := B\cap \{f_{z_j}>c\}$ we have the estimates $g\leq K$ and $f_{z_j}>c$. Hence, $g\ind_{A_j}\leq \frac{K}{c} f_{z_j}$ and, by non-negativity of $U_x$,
\begin{equation}\label{eq:proof_mix6}
\int_{B}  U_x(g\ind_{A_j}) \wedge  g {\rm{d}\mu}
\leq
\int_{B} \left(\frac{K}{c} f_{x+z_j}\right) \wedge K  {\rm{d}\mu}
\overset{}{\underset{x\to\infty}\longrightarrow} 0
\end{equation}
because $\frac{K}{c} f_{x+z_j}\to 0$ locally in measure by assumption (b) which, as we already know, is equivalent to (c).
Writing $g= g \ind_B + g \ind_{S\backslash B}$, we obtain
$$
\int_{S}  (U_x g) \wedge  g {\rm{d}\mu}
\leq
\int_S U_x(g \ind_{S\backslash B}) {\rm{d}}\mu
+
\int_{S} U_{x}(g\ind_B) \wedge g {\rm{d}}\mu.
$$
We have $\int_S U_x(g \ind_{S\backslash B}) {\rm{d}}\mu\leq \eps$ using~\eqref{eq:proof_mix_bd1} and because $U_x$ is $L^1$-isometry. The second integral can be estimated as follows:
$$
\int_{S} U_{x}(g\ind_B) \wedge g {\rm{d}}\mu
\leq
\int_{S\backslash B} g {\rm{d}}\mu +
\int_{B}  U_x (g\ind_B) \wedge  g {\rm{d}\mu}
\leq
\eps + \int_{B}  U_x \left(g\ind_D + \sum_{j=1}^m g \ind_{A_j}\right) \wedge  g {\rm{d}\mu}.
$$
Using the inequality $(a_1+\ldots+a_k)\wedge b \leq a_1\wedge b +\ldots + a_k\wedge b$, we obtain
$$
\int_{S}  U_x (g\ind_B) \wedge  g {\rm{d}\mu}
\leq
\eps + \int_{B}  U_x (g\ind_D) {\rm{d}}\mu
+ \sum_{j=1}^m \int_B U_x(g \ind_{A_j}) \wedge  g {\rm{d}\mu}.
$$
Since $U_x$ is $L^1$-isometry, we have $\int_{B}  U_x (g\ind_D) {\rm{d}}\mu\leq \eps$ by~\eqref{eq:proof_mix_bd2}. Recalling~\eqref{eq:proof_mix6} we obtain that
$$
\limsup_{x\to\infty}  \int_S ((U_x g)\wedge g) {\rm{d}\mu} \leq 3\eps.
$$
Since this is true for every $\eps>0$, the limit is in fact $0$ and we obtain (d).
\end{proof}

\begin{remark}
Condition (d) in Theorem~\ref{theo:mixing_for_flows} can be replaced by the following seemingly stronger one:
For every non-negative functions $g,h\in L^1(S,\mu)$ we have
$$
\lim_{x\to\infty} \int_S ( (U_x g) \wedge h ) {\rm{d}} \mu = 0.
$$
It is clear that this condition implies (d). To see the converse, note that by the non-negativity property of $U_x$,
$$
\int_S (U_x g\wedge h) {\rm{d}}\mu
\leq
\int_S (U_x (g\vee h)\wedge (g\vee h)) {\rm{d}}\mu.
$$
\end{remark}

\subsection{Mixing/non-mixing decomposition}
It is known that the Hopf decomposition can be used to characterize the mixed moving maximum property, whereas Neveu decomposition characterizes ergodicity. In the next proposition we construct a decomposition which characterizes mixing. For measure-preserving maps, this decomposition was introduced by Krengel and Sucheston~\cite{krengel_sucheston,krengel_sucheston_BULL}. E.\ Roy~\cite{roy_ID} used it to characterize mixing of sum-infinitely divisible processes. Note that we consider non-singular flows (which is a broader class than measure preserving flows).
\begin{theorem}\label{theo:mix_decomposition}
Consider a non-singular, measurable flow $(\phi_x)_{x\in\mathcal X}$ acting on a $\sigma$-finite measure space $(S,\cB,\mu)$. There is a decomposition of $S$ into two disjoint measurable sets $S = N_{0} \cup N_{+}$, $N_0\cap N_+ = \varnothing$, such that
\begin{itemize}
\item[(i)] $N_0$ and $N_+$ are $(\phi_x)_{x\in \mathcal{X}}$-invariant, modulo null sets.
\item[(ii)] For every non-negative function $g\in L^1(S, \mu)$ supported on $N_0$,
$$
\lim_{x\to \infty} \int_S (U_x g\wedge g) {\rm{d}\mu} =0.
$$
\item[(iii)] For every nonnegative function $h\in L^1(S, \mu)$ supported on $N_+$ and not vanishing identically,
$$
\limsup_{x\to \infty} \int_S (U_x h\wedge h) {\rm{d}\mu} > 0.
$$
\end{itemize}
Properties (ii) and (iii) define the components $N_+$ and $N_0$ uniquely, modulo null sets.
\end{theorem}
\begin{proof}
Let $\mathcal{H}$ be the family of all measurable sets $A\subset S$ such that $\mu(A)<\infty$ and $U_x \ind_A\to 0$ locally in measure, as $x\to\infty$. By the positivity of $U_x$, the family $\mathcal{H}$ is hereditary, that is it contains with every set $A$ all its measurable subsets. Denote by $N_0$ the measurable union of $\mathcal{H}$; see~\cite[pp.~7--8]{A98} for its existence.

\vspace*{2mm}
\noindent
\textit{Proof of (ii).}  Take any non-negative function $g\in L^1(S,\mu)$ supported on $N_0$.
Fix $\eps>0$. Let $K$ be sufficiently large so that the set $B:=\{g\leq K\}$ satisfies
\begin{equation}\label{eq:proof_mix_bd2_copy0}
\int_{S\bsl B} g {\rm{d}}\mu < {\eps}.
\end{equation}
Let $\delta>0$ be such that for every measurable set $D\subset B$ with $\mu(D) <\delta$ we have $\int_D g{\rm{d}}\mu <\eps$. By the exhaustion lemma~\cite[pp.~7--8]{A98} we can find finitely many sets $A_1,\ldots, A_m\in\mathcal{H}$ such that $\mu(B\bsl \cup_{j=1}^m{A_j}) < \delta$ and hence,
\begin{equation}\label{eq:proof_mix_bd2_copy}
\int_{B\bsl A} g{\rm{d}}\mu < {\eps},
\end{equation}
where we introduced the set $A := A_1\cup \ldots\cup A_m$. For every $j\in\{1,\ldots,m\}$ we have, by the positivity of $U_x$,
\begin{equation}\label{eq:proof_mix6_copy}
\int_{B}  (U_x(g\ind_{A_j\cap B})) \wedge  g {\rm{d}\mu}
\leq
\int_{B} (KU_x(\ind_{A_j\cap B})) \wedge K  {\rm{d}\mu}
\overset{}{\underset{x\to\infty}\longrightarrow} 0
\end{equation}
because $U_x \ind_{A_j\cap B}\to 0$ locally in measure.
We have the estimate
$$
\int_S U_x g\wedge g {\rm{d}\mu}
\leq
\int_{S\bsl B} g {\rm{d}}\mu + \int_{B} (U_xg\wedge g){\rm{d}}\mu
\leq
\eps + \int_{B}  U_x \left(g\ind_{S\backslash (A\cap B)}  + \sum_{j=1}^m g \ind_{A_j\cap B}\right) \wedge  g {\rm{d}\mu}.
$$
Using the inequality $(a_1+\ldots+a_k)\wedge b \leq a_1\wedge b +\ldots + a_k\wedge b$, we obtain
$$
\int_{S}  U_x g \wedge  g {\rm{d}\mu}
\leq
\eps + \int_{B}  U_x (g\ind_{S\bsl (A\cap B)}) {\rm{d}}\mu
+ \sum_{j=1}^m \int_B U_x(g \ind_{A_j\cap B}) \wedge  g {\rm{d}\mu}.
$$
Since $U_x$ is an $L^1$-isometry, we have $\int_{B}  U_x (g\ind_{S\bsl (A\cap B)}) {\rm{d}}\mu\leq 2\eps$ by~\eqref{eq:proof_mix_bd2_copy0} and~\eqref{eq:proof_mix_bd2_copy}. By~\eqref{eq:proof_mix6} we obtain that
$$
\limsup_{x\to\infty}  \int_S U_x g\wedge g {\rm{d}\mu} \leq 3\eps,
$$
which proves (ii) since $\eps>0$ is arbitrary.

\vspace*{2mm}
\noindent
\textit{Proof of (iii).} We argue by contraposition. Assume that a non-negative function $h\in L^1(S, \mu)$ supported on $N_+:=S\bsl N_0$ and not vanishing identically satisfies
$
\lim_{x\to \infty} \int_S (U_x h\wedge h) {\rm{d}\mu} =0.
$
For a sufficiently small $b>0$, the set $A:=\{h>b\}$ has positive,  finite measure, and  (by the positivity of $U_x$) satisfies
$$
\lim_{x\to \infty} \int_S U_x \ind_A \wedge \ind_A {\rm{d}\mu} =0.
$$
Since $U_x$ preserves pointwise minima and is an $L^1$-isometry,  we obtain that for every $x_0\in\mathcal X$,
\begin{equation}\label{eq:tech0}
\lim_{x\to \infty} \int_S (U_x \ind_A) \wedge (U_{x_0}\ind_A) {\rm{d}\mu} =0.
\end{equation}
Since $A\subset N_+$ and $\mu(A)>0$, the definition of $N_0$ implies that the sequence $U_x \ind_A$ does not converge locally in $\mu$-measure, as $x\to\infty$.
Hence, we can find a  measurable set $B\subset S$ with $\mu(B)<\infty$ and $a>0$ such that
\begin{equation}\label{eq:tech1}
\limsup_{x\to\infty} \mu(B\cap \{U_x\ind_A > a\}) >a.
\end{equation}
Let $B_0$ be the measurable union of $\supp U_x \ind_A$, $x\in\mathcal X$. Since replacing $B$ by $B\cap B_0$ does not change the validity of~\eqref{eq:tech1}, we can assume that $B\subset B_0$. By the exhaustion lemma, see~\cite[pp.~7--8]{A98}, we can find finitely many $x_1,\ldots,x_m\in\mathcal X$ and $c>0$ such that the set $B$ is covered, up to a subset of measure at most $a/2$, by the sets $\{U_{x_1}\ind_A>c\},\ldots,\{U_{x_m}\ind_A>c\}$. It follows that for every $x\in\mathcal X$ satisfying $\mu(B\cap \{U_x\ind_A > a\}) \geq  a$ we also have
$$
\mu(\{U_x\ind_A>a\} \cap \{U_{x_i}\ind_A>c\}) > a/(4m)
$$
for at least one $i\in \{1,\ldots,m\}$. But this contradicts~\eqref{eq:tech0}, thus proving (iii).

\vspace*{2mm}
\noindent
\textit{Proof of the uniqueness.} Let $S=\tilde N_0 \cup \tilde N_+$ be another disjoint decomposition enjoying properties (ii) and (iii). If $\mu(N_0\cap \tilde N_+)>0$, then we can find a set $A\subset N_0\cap \tilde N_+$ with $\mu(A)\neq 0,\infty$ (recall that $\mu$ is $\sigma$-finite). The indicator function of this set must satisfy both $\lim_{x\to \infty} \int_S (U_x\ind_A \wedge \ind_A) {\rm{d}\mu} =0$ (because $A\subset N_0$) and $\limsup_{x\to \infty} \int_S (U_x\ind_A \wedge \ind_A) {\rm{d}\mu}>0$ (because $A\subset \tilde N_+$), which is a contradiction. Similarly, the assumption $\mu(\tilde N_0\cap N_+)>0$ leads to a contradiction. Hence, the decompositions $S=N_0 \cup N_+$ and $S=\tilde N_0 \cup \tilde N_+$ coincide modulo $\mu$.

\vspace*{2mm}
\noindent
\textit{Proof of (i).}
We show that the decomposition $S = N_0\cup N_+$ is $(\phi_x)_{x\in \mathcal{X}}$-invariant, modulo null sets.
It is easy to check that for every $y\in\mathcal X$ the decomposition $S = \phi_y(N_0)\cup \phi_y(N_+)$ enjoys properties (ii) and (iii). Indeed, if $g$ is a function supported on $\phi_y(N_0)$, then $U_{y}g$ is supported on $N_0$ and hence,
$$
\lim_{x\to \infty} \int_S (U_x g \wedge g) {\rm{d}\mu} = \lim_{x\to \infty} \int_S U_{y}(U_x g \wedge g) {\rm{d}\mu}= \lim_{x\to \infty} \int_S (U_{x} U_{y} g \wedge U_{y} g) {\rm{d}\mu} =0
$$
by (ii). Similarly, one verifies that $\phi_y(N_+)$ satisfies (iii). The uniqueness of the decomposition implies that $N_0=\phi_y(N_0)$ and $N_+=\phi_y(N_+)$ modulo null sets.
\end{proof}


\begin{remark}
Krengel and Sucheston~\cite{krengel_sucheston} called a measure-preserving flow $(\phi_x)_{x\in\mathbb Z}$ mixing if
$$
\lim_{x\to\infty} \mu(\phi_x A\cap A) = 0
$$
for every set $A\in\mathcal B$ with $\mu(A)<\infty$.  Thus, in the measure-preserving case, the decomposition from Theorem~\ref{theo:mix_decomposition} coincides with the decomposition of Krengel and Sucheston~\cite{krengel_sucheston,krengel_sucheston_BULL}.
\end{remark}

The decomposition introduced in Theorem~\ref{theo:mix_decomposition} characterizes mixing of max-stable processes.
\begin{theorem}\label{theo:mixing_for_flows_decomposition}
Let $\eta$ be a stationary, stochastically continuous max-stable processes with a flow representation~\eqref{eq:flow}. Then $\eta$
is mixing if and only if $N_+=\varnothing \mod \mu$.
\end{theorem}
\begin{proof}
Follows immediately from Theorem~\ref{theo:mixing_for_flows}.
\end{proof}

We can introduce a decomposition of a stationary max-stable process $\eta$ into mixing and non-mixing components as follows: $\eta=\eta_0\vee \eta_+$ with
\[
\eta_0(x) = \int_{N_0}^e f_x(s)M(\mathrm{d}s)
\quad \mbox{and}\quad
\eta_+=\int_{N_+}^e f_x(s) M(\mathrm{d}s),
\quad
x\in\cX.
\]
Clearly, $\eta_0$ and $\eta_+$ are independent stationary max-stable processes. Using argumentation as in the proof of Theorem~2.4 in~\cite{S05} (mapping to the minimal representation), it can be shown that the laws of $\eta_0$ and $\eta_+$ do not depend on the choice of the flow representation.

\subsection{An open question}
We have provided characterizations of the null recurrent and the dissipative components of a max-stable process in terms of its  spectral functions, see condition~(f) in Theorem~\ref{theo:ergodicity} and conditions~(c)-(d) in Theorem~\ref{theo:MMM}. This allows us to obtain the positive/null and conservative/dissipative decompositions of a max-stable process given by de Haan representation \eqref{eq:deHaan} directly via cone decompositions (see Proposition~\ref{prop} and Corollary~\ref{cor}). We have also provided a new decomposition into mixing/non mixing components. It is  therefore natural to ask whether a pathwise characterization of this decomposition is available. In view of the equivalence (e)-(f) in Theorem~\ref{theo:ergodicity}, we can wonder whether mixing can be characterized by the condition
\begin{equation}\label{eq:mixing_question}
\liminf_{x\to\infty} Y(x) = 0 \quad \text{a.s.}
\end{equation}
The answer is negative. Although mixing implies~\eqref{eq:mixing_question} (because mixing is equivalent to $Y(x) \to 0$ in probability  which implies a.s.\ convergence to $0$ along a subsequence), the converse is not true. We shall show that a counterexample is provided by a process constructed in~\cite{KS10}.

Consider a max-stable process $\eta(t)= \vee_{i=1}^{\infty} U_i Y_i(t)$ as in~\eqref{eq:deHaan}, where the spectral functions $(Y_i)_{i\in\bbN}$ are i.i.d.\ copies of the log-normal process
\begin{equation}\label{eq:def_Brown_Resnick_Y}
Y(t) = \exp\left\{Z(t) - \frac 12 \sigma^2(t)\right\}, \quad t\in\bbR,
\end{equation}
with $(Z(t))_{t\in \mathbb R}$  a zero-mean Gaussian process with stationary increments, $Z(0)=0$, and incremental variance
$$
\sigma^2(t) := \Var (Z(s+t)-Z(s)) = \sum_{k=1}^{\infty} \left(1-\cos\left(\frac{2\pi t}{2^k}\right)\right).
$$
An explicit series representation of $(Z(t))_{t\in \mathbb R}$ is given by
\[
Z(t)=\frac{1}{\sqrt{2}}\sum_{k=1}^\infty \left(N_k'\left(1-\cos \frac{2\pi t}{2^k}\right)+N_k''\sin \frac{2\pi t}{2^k} \right),
\]
where $N_k',N_k''$, $k\in\mathbb N$, are independent standard normal random variables.
The max-stable process $\eta$ belongs to the family of the so-called Brown--Resnick processes and is stationary; see \cite{KSdH09}.
\begin{proposition}
The max-stable process $\eta$ is ergodic but non-mixing although it satisfies~\eqref{eq:mixing_question}.
\end{proposition}
\begin{proof}
The fact that $\eta$ is ergodic but non-mixing was proven in~\cite{KS10}. We show here that Equation \eqref{eq:mixing_question} is satisfied.
It was shown in~\cite{KS10} that there is a sequence $x_1<x_2<\ldots\to+\infty$ such that $\lim_{n\to\infty} \sigma^2(x_n) = +\infty$. Passing, if necessary, to a subsequence, we can assume that $\sigma^2(x_n) > n^2$. For every $\eps \in (0,1)$ we have
$$
\bbP[Y(x_n) > \eps]
=
\bbP\left[Z(x_n) > \log \eps + \frac{1}{2}\sigma^2(x_n)\right]
=
\bbP\left[N > \frac{\log \eps}{\sigma(x_n)} + \frac{1}{2}\sigma (x_n)\right],
$$
where $N$ is a standard normal random variable. It follows that
$$
\sum_{n=1}^{\infty} \bbP[Y(x_n) > \eps]  \leq \sum_{n=1}^{\infty} \bbP\left[N > \frac{n}{2} + \log \eps\right] < \infty.
$$
By the Borel--Cantelli lemma, the probability that only finitely many events $\{Y(x_n)>\eps\}$ occur equals $1$. Since this holds for every $\eps\in (0,1)$,
we obtain that $\lim_{n\to\infty} Y(x_n) = 0$ a.s.\ and this implies \eqref{eq:mixing_question}.
\end{proof}


\bibliographystyle{plain}
\bibliography{Biblio}

\end{document}